\documentclass[11pt,a4paper]{article}
\usepackage[english]{babel}
  \usepackage{amsmath,amsfonts,amssymb,relsize,amsthm,chngcntr}
   \usepackage{amsfonts,amsmath,graphicx,color,epsfig}
  \usepackage{esint,authblk}

  \usepackage{hyperref} 

 \usepackage[active]{srcltx} 
 \usepackage{color,soul} 

\newtheorem{theorem}{Theorem}[section]

\newtheorem{proposition}[theorem]{Proposition}

\theoremstyle{definition}

\newtheorem{remark}{Remark}

\numberwithin{equation}{section}

\def\O{{\Omega}}

\def\eps{{\varepsilon}}

\def\m{{\mathcal{M}}}

\def\M{{\mathcal{M}}}

\def\R{{\mathbb{R}}}

\def\N{{\mathbb{N}}}

\newenvironment{formula}[1]{\begin{equation}\label{eq:#1}}
                       {\end{equation}\noindent}

\def\Fi#1{\begin{formula}{#1}}
\def\Ff{\end{formula}\noindent}

\setstcolor{red}

\def\ds{\displaystyle}

\title{
Equilibrium and sensitivity analysis \\ of a spatio-temporal host-vector epidemic model
}

\author[$\dagger$,1]{Olivier Martin}
\author[1]{Yasmil Fernandez-Diclo}
\author[1]{J\'er\^ome Coville}
\author[1]{Samuel Soubeyrand}

\affil[1]{INRAE, BioSP, 84914 Avignon, France}
\affil[$\dagger$]{olivier.martin@inrae.fr}

\begin{document}

\maketitle

\begin{abstract}
Insect-borne diseases are diseases carried by insects affecting humans, animals or plants. 
They have the potential to generate massive outbreaks such as the Zika epidemic in 2015-2016 mostly distributed in the Americas, the Pacific and Southeast Asia, and the multi-foci outbreak caused by the bacterium {\it Xylella fastidiosa} in Europe in the 2010s. In this article, we propose and analyze the behavior of a spatially-explicit compartmental model adapted to pathosystems with fixed hosts and mobile vectors disseminating the disease. The behavior of this model based on a system of partial differential equations is complementarily characterized via a theoretical study of its equilibrium states and a numerical study of its transitive phase using global sensitivity analysis. The results are discussed in terms of implications concerning the surveillance and control of the disease over a medium-to-long temporal horizon.

\

\end{abstract}


\section{Introduction}

A large class of diseases are indirectly transmitted between hosts via insects, which play the role of vectors transporting the pathogens causing the diseases of interest from infectious hosts to susceptible hosts. For instance, malaria, Zika and dengue fever are transmitted by mosquitoes, Lyme disease by ticks, sharka by aphids, and Pierce's disease by xylem-feeding leafhoppers. For some of these examples, mathematical dynamic models have provided insights into how to improve disease control, potentially leading to disease eradication over large spatial territories and time periods; see e.g. \cite{reiner2013}. 

In this article, we are interested in a spatially-explicit compartmental model adapted to pathosystems with fixed hosts (typically, plants) and mobile vectors disseminating the disease. Compartmental models describe the dynamics of population fractions in specific disease states such as susceptible, exposed, infectious and recovered. They have been exploited to derive properties of idealized pathosystems \cite{diekmann2000,kermack1927,murray2003}, to search for efficient surveillance, control or eradication strategies \cite{kyrkou2018,rimbaud2019}, to infer epidemiological parameters, reconstruct past dynamics and predict disease propagation \cite{abboud2019,britton2014,soubeyrand2018}.  

The specification of the model considered in this article was partly driven by the case of {\it Xylella fastidiosa}, a bacterium which is pathogenic for a large range of plants and transmitted from infectious plants to susceptible plants via xylem-feeding leafhoppers \cite{purcell2013}. This plant pathogen was recently detected in southeastern France (in July 2015) and has the potential to spread beyond its current spatial distribution \cite{abboud2019,denance2017,godefroid2019,martinetti2019}. We built a model grounded on differential equations and explicitly handling both the host population and the vector population. This model will be used in further studies as a basis for estimating epidemiological parameters from surveillance data and assessing diverse control strategies, which may target the hosts, the vectors or both agents. However, to be able to properly interpret the output of these future analyses, we investigate in this article the properties of the above-mentioned model. We specifically aim to understand the impact of parameters on the behaviour of the model, in particular its equilibrium states, if any, and its transitive phase. 

In what follows, we present the model and derive its equilibrium states in Section \ref{sec:model}. The theoretical analysis of equilibrium states is made in two contexts: (i) when the vector population is considered as permanent, and (ii) when the vector population has a cyclic annual dynamics consisting of an emergence stage at the beginning of the year, a mortality stage at the end of the year and no adult-to-offspring transmission of the pathogen from one year to the following one. The latter context likely corresponds to the situation of vectors of {\it Xylella fastidiosa} in France. In Section \ref{sec:AS}, we numerically explore the impact of parameters on the transitive phase of the model by adapting tools of sensitivity analysis \cite{saltelli2000,saltelli2008} to the spatio-temporal framework that we deal with. Finally, we discuss implications of our results in Section \ref{sec:discu}.










\section{A vector-host epidemic model and its equilibrium states}\label{sec:model}

Partial differential equations are common tools for modeling biological invasions \cite{okubo1980,Shigesada1997}. Hereafter, we focus on the invasion of a pathogen in a population of fixed hosts (plants for example) that is transmitted by vectors (insects for example), and we propose compartmental models detailing the transmission process, which is at the core of any epidemiological model of infectious diseases. 

\subsection {A model $\mathcal{M}_1$ with coupled partial differential equations}

The following epidemic model is based on coupled partial differential equations
(PDEs), describing the interaction between the hosts
and the vectors. The PDE system, denoted $\mathcal{M}_1$, consists of two epidemiological sub-models indexed by time and space: a Susceptible-Exposed-Infected (SEI) model
for the hosts and a Susceptible-Infected (SI) model for the vectors. Let $S_h(t,x)$, $E_h(t,x)$ and $I_h(t,x)$ be the numbers of susceptible, exposed and infected hosts, respectively, at time $t>0$ and location $x \in \Omega$, where $\Omega\subset\R^d$ is the
studied spatial domain ($d\in\N^*$; typically $d=2$). Let $S_v(t,x)$ and $I_v(t,x)$ be the numbers of
susceptible and infected vectors. The PDE system is specified as follows:
\begin{align}
&\partial_t S_h(t,x)=-\beta_v(x)S_h(t,x)I_v(t,x)\label{eq1}\\
&\partial_t E_h(t,x)=\beta_v(x)S_h(t,x)I_v(t,x) -\eps E_h(t,x)\label{eq2}\\
&\partial_t I_h(t,x)=\eps E_h(t,x)\label{eq3}\\
&\partial_t S_v(t,x)= \Delta (D(x) S_v(t,x)) -\beta_h(x)S_v(t,x)I_h(t,x)\label{eq4}\\
&\partial_t I_v(t,x)= \Delta(D(x)I_v(t,x)) +\beta_h(x)S_v(t,x)I_h(t,x)\label{eq5}
\end{align}
with the following boundary and initial conditions:
\begin{align}
&\partial_n(D(x)S_v(t,x))=\partial_n(D(x)I_v(t,x))=0 \qquad \text{ for all } t>0, x \in \partial \O\label{eq-bc}\\
&(S_h(0,x),E_h(0,x),I_h(0,x), S_v(0,x),I_v(0,x))=(S_h^0(x),E_h^0(x),I_h^0(x), S_v^0(x),I_v^0(x)) \qquad \text{ for all }  x \in  \O, \label{eq-ci}
\end{align}
where $S_h^0$, $E_h^0$, $I_h^0$, $S_v^0$ and $I_v^0$ are spatial functions to be specified, parameter $\beta_v$ gives the contact rate (number of contacts per unit of time) of a vector with hosts, $\beta_h$ the contact rate  of a host with vectors, $D(x)$ is the coefficient of diffusion of vectors at location $x$, and $\epsilon^{-1}$ is the expected duration of the exposed (i.e. latency) period.

Note that by construction, there exists a constant $C^*>0$ and a spatial function $N$ such that for all times $t>0$:
\begin{align}
S_h(t,x)+E_h(t,x)+I_h(t,x)=S_h(0,x)+E_h(0,x)+I_h(0,x)=N(x)\label{invar1},\\
\int_{\O} (S_v(t,x)+I_v(t,x))\, dx= \int_{\O} (S_v(0,x)+I_v(0,x))\, dx= C^*. \label{invar2} 
\end{align}

\begin{remark}
These two invariant quantities are, respectively, the total number of hosts at $x$ ($N(x)$) and the total number of vectors in $\Omega$ $(C^*)$. 
\end{remark}

Note also that up to a redefinition of the function $S_v(t,x), I_v(t,x)$ and $\beta_v(x)$  by 
$s_v(t,x)=D(x)S(t,x), i_v(t,x)=D(x)I(t,x)$ and $\bar \beta_v(x)=\frac{\beta_v(x)}{D(x)}$, for $t>0$ and $x\in \O$ the system   \eqref{eq1}-\eqref{eq5}  can be reformulated as 

\begin{align}
&\partial_t S_h(t,x)=-\bar\beta_v(x)S_h(t,x)i_v(t,x)\label{eq1'}\\
&\partial_t E_h(t,x)=\bar\beta_v(x)S_h(t,x)i_v(t,x) -\eps E_h(t,x)\label{eq2'}\\
&\partial_t I_h(t,x)=\eps E_h(t,x)\label{eq3'}\\
&\partial_t s_v(t,x)= D(x)\Delta s_v(t,x) -\beta_h(x)s_v(t,x)I_h(t,x)\label{eq4'}\\
&\partial_t i_v(t,x)= D(x)\Delta i_v(t,x) +\beta_h(x)s_v(t,x)I_h(t,x)\label{eq5'}
\end{align}
with the following boundary and initial conditions:
\begin{align}
&\partial_n s_v(t,x)=\partial_n i_v(t,x)=0 \qquad \text{ for all } t>0, x \in \partial \O\label{eq-bc'}\\
&(S_h(0,x),E_h(0,x),I_h(0,x), s_v(0,x),i_v(0,x))=(S_h^0(x),E_h^0(x),I_h^0(x), \frac{S_v^0(x)}{D(x)},\frac{I_v^0(x)}{D(x)}) \qquad \text{ for all }  x \in  \O, \label{eq-ci'}
\end{align}

\subsection{A reduced version of the model $\mathcal{M}_1$}

By using Equation \eqref{invar1} and introducing the reduced variables $s_h=\frac{S_h}{N}$ and $i_h=\frac{I_h}{N}$, for $t>0$ and $x \in \O$ we can rewrite the model \eqref{eq1'}--\eqref{eq5'} in the following way:
\begin{align}
&\partial_t s_h(t,x)=-\bar\beta_v(x)s_h(t,x)i_v(t,x)\label{eq-red1}\\
&\partial_t i_h(t,x)=\eps (1-i_h(t,x)-s_h(t,x))\label{eq-red2}\\
&\partial_t s_v(t,x)= D(x) \Delta s_v(t,x) -\beta_h(x)N(x)s_v(t,x)i_h(t,x)\label{eq-red4}\\
&\partial_t i_v(t,x)= D(x)\Delta i_v(t,x) +\beta_h(x)N(x)s_v(t,x)i_h(t,x)\label{eq-red5}
\end{align}

Since $s_h$ satisfies \eqref{eq-red1}, by integrating with respect to time we get:
\begin{equation}\label{red1}
s_h(t,x)=s_h(0,x)\, \exp\left(-\bar \beta_v(x)\ds{\int_0^t i_v(\tau,x)\,d\tau}\right).
\end{equation}
Furthermore, by integrating \eqref{eq-red2}, we also deduce that: 
 $$
i_h(t,x)=(1-e^{-\eps t})+i_h(0,x)e^{-\eps t} -\eps \int_{0}^t e^{\eps(\tau-t)}s_h(\tau,x)\,d\tau.
$$
Thus, plugging \eqref{red1} in the above equation we end up with:
\begin{equation}\label{red2}
i_h(t,x)=(1-e^{-\eps t})+i_h(0,x)e^{-\eps t} -\eps s_h(0,x)\int_{0}^t e^{\eps(\tau-t)}\exp\left({-\bar\beta_v(x)\,\ds{\int_0^\tau i_v(\tau',x)\,d\tau'}}\right)\,d\tau,
\end{equation}
which in turn leads to the following coupled system of PDE (by plugging \eqref{red2} in \eqref{eq-red4} and \eqref{eq-red5}, and by using \eqref{invar1}):
\begin{multline}
\partial_t s_v(t,x) -D(x) \Delta  s_v(t,x)=\\-N(x)\beta_h(x)s_v(t,x)\left( 1-s_h(0,x)e^{-\eps t} -\eps s_h(0,x)\int_{0}^t e^{\eps(\tau-t)}\,e^{-\bar\beta_v(x)\,\ds{\int_0^\tau i_v(\tau',x)\,d\tau'}}\,d\tau\right)\label{eq-red6}       
\end{multline}
\begin{multline*}
\partial_t i_v(t,x) -D(x) \Delta i_v(t,x)=\\ N(x)\beta_h(x)s_v(t,x)\left( 1-s_h(0,x)e^{-\eps t} -\eps s_h(0,x)\int_{0}^t e^{\eps(\tau-t)}\,e^{-\bar\beta_v(x)\,\ds{\int_0^\tau i_v(\tau',x)\,d\tau'}}\,d\tau\right).   
\end{multline*}
Finally, by adding the two equations we can check that $s_v(t,x)+i_v(t,x)$ satisfies the following  
standard diffusion equation: 
  \begin{align}
&\partial_t a(t,x) -D(x) \Delta a(t,x)=0 \quad \forall t>0, x \in \O    \label{eq-a}   \\
&\partial_n a(t,x)=0 \quad \forall t>0, x \in \partial\O\label{eq-a-bc}\\
&a(0,x)=s_v(0,x)+i_v(0,x) \quad \forall x \in \O.  \label{eq-a-ci} 
\end{align}
Thus,  by introducing the following notation:
$$f(t,x,i_v(t,x)):= N(x)\beta_h(x)\left(1-s_h(0,x)e^{-\eps t} -\eps s_h(0,x)\int_{0}^t e^{\eps(\tau-t)}\,e^{-\bar\beta_v(x)\,\ds{\int_0^\tau i_v(\tau',x)\,d\tau'}}\,d\tau\right),$$
 we can further reduce the  system \eqref{eq1'}--\eqref{eq5'} to the following single equation:
\begin{equation}\label{ultim}
\partial_t i_v(t,x) -D(x) \Delta i_v(t,x)=(a(t,x)-i_v(t,x))f(t,x,i_v(t,x)), 
\end{equation} 
where the function $a$ is the solution of the diffusion-equation system \eqref{eq-a}--\eqref{eq-a-ci}.


\subsection{Analysis of the system $\mathcal{M}_1$}

We first observe that for a given positive pair $(N(x),C^*)$ (i.e. $N(x)\ge 0, C^*>0$), the system $\mathcal{M}_1$ has only two positive equilibria that satisfy the invariance conditions \eqref{invar1} and \eqref{invar2}. Moreover, one solution is globally unstable and the other one is globally stable. Namely, we have:
\begin{proposition}\label{prop1}
Let $\O\subset \R^d$ be a bounded smooth domain (at least $C^1$) and let $N\in C(\bar\O)$ be a positive function and $C^*$ a positive constant, let us also denote $|\O|_\mu$ the measure of $\O$ with respect to the positive measure $d\mu=\frac{dx}{D(x)}$. Then $(N(x),0,0,\frac{C^*}{|\O|_{\mu}D(x)},0)$ and $(0,0,N(x),0,\frac{C^*}{|\O|_{\mu}D(x)})$
 are the only non negative stationary solution of the system \eqref{eq1}--\eqref{eq5} satisfying the invariance conditions \eqref{invar1} and \eqref{invar2} and the boundary condition \eqref{eq-bc}. Moreover the stationary state $(N(x),0,0,\frac{C^*}{|\O|_\mu D(x)},0)$ is globally unstable whereas the state $(0,0,N(x),0,\frac{C^*}{|\O|_\mu D(x)})$ is globally stable.
\end{proposition}

 The proof of this proposition uses rather standard elementary analysis, which can be found in the appendix section.
Next, we derive an important convergence property of the system. Namely, we show  the exponential convergence of the trajectories to its equilibria. 
\begin{proposition} \label{prop2}
Let $\O\subset \R^d$ be a bounded smooth domain and let $(S_h,E_h,I_h, S_v,I_v)$ be a solution of the system \eqref{eq1}--\eqref{eq5} with boundary  \eqref{eq-bc} and a smooth initial condition $(S_h^0(x),E_h^0(x),I_h^0(x), s_v^0(x),i_v^0(x))$. Then there exits positive constants $C$ and $\lambda$ such that    
\begin{itemize}
\item[i)] $\ds{S_h(t,x)\le Ce^{-\lambda t}}$, 
\item[ii)] $\ds{E_h(t,x)\le Ce^{-\lambda t}}$,
\item[iii)] $\ds{s_v(t,x)\le Ce^{-\lambda t}}$,
\item[iv)] $\ds{\|I_h-S_h^0-E_h^0-I_h^0\|_{\infty} \le Ce^{-\lambda t}}$,
\item[v)] $\ds{\left\|i_v-\frac{\int_{\O} s_v^0(x)+i_v^0(x)\,dx}{|\O|}\right\|_{2} \le Ce^{-\lambda t}}$.
\end{itemize}
\end{proposition}
\begin{proof}
Observe that thanks to \eqref{red1}, \eqref{red2} and \eqref{invar1}, we can deduce  $ii)$ and $iv)$ from $i)$ and $v)$. Thus, we only have to prove $i)$, $iii)$ and $v)$. To prove such behaviour note that it is sufficient to  show that $iii)$ and $v)$ holds as well for the redefined function $s_v(t,x)$ and $i_v(t,x)$.
We will see  also that $v)$ is a consequence of $iii)$. Indeed,  let us look further at the properties of $f$ and let us recall that $N(x):=S_h(x)^0+E_h(x)^0+I_h^0(x)$. 

 We can easily check that for all $t\ge 0$ and $x$, we have
 $$N(x)\beta_h(x)i_0(x)\le f(t,x,i_v(t,x))\le N(x)\beta_h(x).$$
 Thus going back to  \eqref{eq-red6}, we deduce from the above inequality and a straightforward application of the parabolic maximum principle  that
 $$S_{m}(t,x)\le s_v(t,x)\le S_{M}(t,x), $$ 
 where $S_{m}$ and $S_{M}$ respectively satisfy:
 
  \begin{align}\label{esti1}
&\partial_t S_m(t,x) -D(x)\Delta S_m(t,x)=S_m(t,x)N(x)\beta_h(x)\quad \forall t>0, x \in \O\\
&\partial_n S_m(t,x)=0 \quad \forall t>0, x \in \partial\O\\
&S_m(0,x)=s_v^0(x) \quad \forall x \in \O   ,
\end{align}
 
and  
  \begin{align}\label{esti2}
&\partial_t S_M(t,x) -D(x)\Delta S_M(t,x)=S_M(t,x)i_0(x)N(x)\beta_h(x)\quad \forall t>0, x \in \O \\
&\partial_n S_M(t,x)=0 \quad \forall t>0, x \in \partial\O\\
&S_M(0,x)=s_v^0(x) \quad \forall x \in \O   .
\end{align}
 
From standard parabolic theory, we know that 
$$ S_M(t,x) \le C(s^0_v)e^{-\lambda_1t}\varphi_1(x), $$
with $(\lambda_1,\varphi_1)$ solution of the spectral problem
  \begin{align}\label{spec1}
&D(x)\Delta \varphi_1(x)-i_0(x)N(x)\beta_h(x)\varphi_1(x)+\lambda_1\varphi_1(x)=0 \quad \text{for} \quad x \in \O\\
&\partial_n \varphi_1(x)=0 \quad \forall  x \in \partial\O   .
\end{align}
 
Note that the exponential behaviour on $S_M$ implies that  $iii)$ holds. 

\begin{remark}
$\lambda_1$ can be expressed through some various equivalent variational formula. In particular, for the positive measure $d\mu(x)=\frac{dx}{D(x)}$ we have 
$$\lambda_1:=\inf_{\varphi \in H^1(\O)}\frac{\ds{\int_{\O}D(x)|\nabla \varphi(x)|^2\,d\mu(x) +\int_{\O} i_0(x)N(x)\beta_h(x)\varphi^2(x)\,d\mu(x)}}{\ds{\int_{\O} \varphi^2(x)\,d\mu(x)}}. $$  
From this variational formula,  we can clearly see the monotone dependence of $\lambda_1$ with respect to the parameter $i_0,N(x)$ and $\beta_h$ but  the dependence of $\lambda_1$ with respect to the diffusion $D(x)$ is still unclear since the measure $d\mu$ depends on $D(x)$. 
When $D(x)$ is a constant, then the above formulation can be simplified. Namely, 
$$\lambda_1=\inf_{\varphi \in H^1(\O)}\frac{\ds{D\int_{\O}|\nabla \varphi(x)|^2\, dx +\int_{\O} i_0(x)N(x)\beta_h(x)\varphi^2(x)\, dx}}{\ds{\int_{\O} \varphi^2(x)\,dx}}.$$
In this situation, we can clearly see the monotone dependence of $\lambda_1$ with respect to the parameter $D$.
\end{remark}

Now, on the one hand  from \eqref{invar2}, we deduce that 
$$\|i_v(t,x) -a(t,x)\|_{\infty}\le C(s^0_v)e^{-\lambda_1t}.$$
On the other hand   since $a$ satisfies the heat equation with homogeneous Neumann boundary condition, we can easily check that $v(t,x)=a(t,x)-\frac{1}{|\O|_\mu}\int_{\O}s_v(0,x)+i_v(0,x)\,d\mu(x)$, satisfies 
\begin{align}
&\partial_t v(t,x) -D(x)\Delta v(t,x)=0 \label{eq-v}\\
&\partial_n v(t,x)=0 \quad \text{ for all }\quad t>0, x\in \partial \O\\
&v(0,x)=s_v(0,x)+i_v(0,x)-\frac{1}{|\O|_{\mu}}\int_{\O}s_v(0,x)+i_v(0,x)\,d\mu(x)\\
&\int_{\O}v(t,x)\,d\mu(x)=0,\quad \text{ for all }\quad t>0 .
\end{align}

So by multiplying the equation \eqref{eq-v} by $v$ and integrating over $\O$ with respect to the measure $\mu$, we get, after integrating by part, 
$$ \partial_t \int_{\O}v^2(t,x)\,d\mu(x)= -2\int_{\O}D(x)|\nabla v(t,x)|^2\, d\mu(x)\le -2\min_{x\in \O}D(x)\int_{\O}|\nabla v(t,x)|^2\, d\mu(x),$$
which  by using a Poincare-Writtinger  inequality yields 
$$ \partial_t \int_{\O}v^2(t,x)\,d\mu(x)\le -2D_\text{min}\lambda_2\int_{\O}|v(t,x)|^2\, d\mu(x),$$
which after integration  in time enforces 
$$\int_{\O}v^2(t,x)\,d\mu(x)\le C(v^0)e^{-2D_\text{min}\lambda_2 t}, $$
where $D_\text{min}:=\min_{x\in \O} D(x)>0$.
Therefore, thanks to \eqref{invar2}
\begin{align*}
\left\|i_v-\frac{\int_{\O} s_v^0(x)+s_v^0(x)\,d\mu(x)}{|\O|_{\mu}}\right\|_{2} &\le \left\| a(t,x)-\frac{\int_{\O} s_v^0(x)+i_v^0(x)\,d\mu(x)}{ |\O|_{\mu}} \right\|_{2}+\|S_v(t,x)\|_{2} \\
& \le \|v\|_{2}+\|S_v(t,x)\|_{2}\le Ce^{-\lambda t}.
\end{align*}

At last, let us prove $i)$. By \eqref{red1} and using that $i_v(t,x)=a(t,x)-s_v(t,x)$ with $a$ defined by \eqref{eq-a}--\eqref{eq-a-ci}, we have 

$$S_h(t,x)=N(x)C(S_h^0)e^{-\bar \beta_v(x)\int_{0}^ti_v(s,x)\,ds}= N(x)C(S^0_h)e^{-\bar \beta_v(x)\int_{0}^ta(s,x)\,ds} e^{\bar \beta_v(x)\int_{0}^ts_v(s,x)\,ds}. $$
Let us estimate both exponential separately  and for simplicity  set the notation 
\begin{align}
&g(t,x):=e^{-\bar \beta_v(x)\int_{0}^ta(s,x)\,ds}\label{exp1}\\
&h(t,x):=e^{\bar \beta_v(x)\int_{0}^ts_v(s,x)\,ds} \label{exp2}
\end{align}

First let us observe that $iii)$ yields a straightforward bounded estimate on $h$ (i.e \eqref{exp2}). Indeed, thanks to $iii)$ we get 
$$
 h(t,x)=e^{\bar \beta_v(x)\int_{0}^ts_v(s,x)\,ds}\le e^{\frac{C(S_v^0)\bar \beta_v(x)}{\lambda_1}[1-e^{-\lambda_1 t}]}. $$
 
Next let us estimate the function $g(t,x)$ (i.e \eqref{exp1}).  Recalling that  $a$ is a positive solution of the heat equation with Neumann boundary condition, we can use  the Krylov-Safonov Harnack inequality up to the boundary (see \cite{Evans1998}) and  for all $\tau>0$ there exists $C(\tau)>0$ such that for all $t>0$ and $x\in \O$
$$\max_{x\in \O} a(t,x)\le C(\tau) \min_{\O} a(t+\tau,x).  $$

From there, thanks to the mass invariance of $a$ with respect to the measure $d\mu(x)$, we get 

$$\frac{1}{|\O|_{\mu}}\int_{\O}[s_v^0(x)+i_v^0(x)]\,d\mu(x)=\frac{1}{|\O|_{\mu}}\int_{\O}a(t,x)\,d\mu(x)\le \max_{x\in \O} a(t,x)\le C(\tau) a(t+\tau,x), $$
which thanks to the definition of $g$  yields   
\begin{align*}
g(t,x):=e^{-\bar \beta_v(x)\int_{0}^ta(s,x)\,ds}&\le  e^{-\bar \beta_v(x)\int_{0}^{t-\tau} a(s+\tau,x)\,ds}\\
&\le e^{-\frac{\bar \beta_v(x) C^*}{C(\tau)|\O|_{\mu}}(t-\tau)}.
\end{align*}
Hence, we get 
$$S_h(t,x)\le N(x)C(S_h^0)e^{\frac{C(S_v^0)\bar \beta_v(x)}{\lambda_1}[1-e^{-\lambda_1 t}]}e^{-\frac{\bar \beta_v(x) C^*}{C(\tau)|\O|_{\mu}}(t-\tau)}. $$

\end{proof}

\subsection{A multi-annual model $\mathcal{M}_2$ with periodic vector emergence and death} \label{sec:multi-annual.model}

Suppose that, within the life cycle of the vector, the pathogen is not transmitted to the offspring. Then, a more realistic model of the pathogen dynamics can be achieved by including a pulse-like component where the vector are reset at some specific periodic time.
This framework developed in \cite{Mailleret2009} translates in the
above model $\mathcal M_1$ as follows:

For all $n\in\N$,  $ nT<t<(n+1)T$ and $x \in  \O$, the quantity $(S_h(t,x,n+1),E_h(t,x,n+1),I_h(t,x,n+1), S_v(t,x,n+1),I_v(t,x,n+1)) $ is assumed to satisfy:
\begin{align}
&\partial_t S_h(t,x,n+1)=-\beta_v(x)S_h(t,x,n+1)I_v(t,x,n+1)\label{eq-imp1}\\
&\partial_t E_h(t,x,n+1)=\beta_v(x)S_h(t,x,n+1)I_v(t,x,n+1) -\eps E(t,x,n+1)\label{eq-imp2}\\
&\partial_t I_h(t,x,n+1)=\eps E_h(t,x,n+1)\label{eq-imp3}\\
&\partial_t S_v(t,x,n+1)= \Delta(D(x)S_v(t,x,n+1)) -\beta_h(x)S_v(t,x,n+1)I_h(t,x,n+1)\label{eq-imp4}\\
&\partial_t I_v(t,x,n+1)= \Delta(D(x)I_v(t,x,n+1)) +\beta_h(x)S_v(t,x,n+1)I_h(t,x,n+1), \label{eq-imp5}
\end{align}
with the boundary conditions 
\begin{equation}
\label{eq-imp-bc}
\partial_n(D(x)S_v(t,x,n+1))=\partial_n(D(x)I_v(t,x,n+1))=0 \qquad \text{ for all } nT<t<(n+1)T, x \in \partial \O.
\end{equation}
Next we have to describe the impulsive condition:
for all $x \in \O$,
\begin{equation} \label{eq-imp-compat}
\begin{cases}
&S_h(nT,x,n+1)=S_h((n+1)T^-,x,n),\\
&E_h(nT,x,n+1)=E_h((n+1)T^-,x,n),\\
&I_h(nT,x,n+1)=I_h((n+1)T^-,x,n),\\
&S_v(nT,x,n+1)= S_v^0(x),\\
&I_v(nT,x,n+1))=0.
\end{cases} 
\end{equation}
Note that the impulsive condition implies the time continuity of $(S_h,E,I_h)$. Moreover, for all times $t>0$ we have:
\begin{align}\label{invar3}
&S_h(t,x)+E(t,x)+I_h(t,x)=S_h(0,x)+E(0,x)+I_h(0,x)=N(x),\\
&\int_{\O}(S_v(t,x)+I_v(t,x))\, dx = \int_{\O}(S_v(0,x)+I_v(0,x))\, dx=C^*.\label{invar4}
\end{align}

Like for model $\mathcal{M}_1$, we may wonder if this system admits one or more  non-negative equilibria. Note that due to the impulsive nature of the system, an equilibrium of $\m_2$  is then  a non negative time periodic  function of period $T$ that solves the system of equation \eqref{eq-imp1}--\eqref{eq-imp5} with boundary condition \eqref{eq-imp-bc}. We can check that the stationary solution $(N(x),0,0,\frac{C^*}{D(x)|\O|_{\mu}},0)$ of the system $\mathcal{M}_1$ is also an equilibrium of $\M_2$ which remains globally unstable. In contrast, the stationary solution $(0,0,N(x),0,\frac{C^*}{D(x)|\O|_{\mu}})$ for $\mathcal M_1$ is not a solution for $\M_2$, since it does not include the impulsive term. We  may still wonder if other  equilibria exist. 
In this aim, we can show:

\begin{proposition}\label{prop3}
Let $\O\subset \R^d$ be a bounded smooth domain (at least $C^1$), let $N\in C(\bar\O)$ and $S_v^0$ be two non-negative functions, and let
 $(S_v^*(t,x),I_v^*(t,x))$ be the solution of
 \begin{align*}
 &\partial_t S_v^*(t,x)= \Delta(D(x) S_v^*(t,x)) -\beta_h(x)S_v^*(t,x)N(x)&\qquad \text{ for all } 0<t<T, x \in  \O \\
&\partial_t I_v^*(t,x)= \Delta(D(x) I_v^*(t,x)) +\beta_h(x)S_v^*(t,x)N(x)&\qquad \text{ for all }  0<t<T, x \in \ \O\\
&S_{v}^*(0,x)=S^0_v(x), \quad I_{v}^*(0,x)=0&\qquad \text{ for all }  x \in  \O\\
&\partial_n(D(x)S_v^*(t,x))=\partial_n(D(x)I_v^*(t,x))=0 &\qquad \text{ for all } 0<t<T, x \in \partial \O.
\end{align*}
 Then,  the state $(0,0,N(x),S^*(t,x),I^*(t,x))$ is the only non-negative equilibrium of the impulsive system \eqref{eq-imp1}--\eqref{eq-imp5} with the impulsive condition \eqref{eq-imp-compat} satisfying the invariance conditions \eqref{invar3} and \eqref{invar4} and the boundary condition \eqref{eq-imp-bc}. 
\end{proposition}   

Like for the analysis of the equilibrium for the system $\M_1$, this proposition is rather standard and its proof is provided in the appendix.

\subsection{Comment on the diffusion specification}

In the models considered above, we have assumed that the diffusion operator that describes the dispersion process of the vector population reflects, at the macroscopic level, an unbiased random walk in a spatial heterogeneous medium; see for example \cite{Turchin1998} for a standard derivation. Other formulations are possible depending on the reality of the studied phenomenon and the choice of the modeler. 
In general, a new formulation impacts the equilibrium analysis. However,  results presented above holds true, up to minor changes, for one of the formulations commonly used to describe the diffusion of a population and grounded on flux consideration combined with some conservation laws such as Fick's law.
This formulation can be written as follows:
\begin{align*}
&\partial_t S_h(t,x)=-\beta_v(x)S_h(t,x)I_v(t,x)\\
&\partial_t E_h(t,x)=\beta_v(x)S_h(t,x)I_v(t,x) -\eps E_h(t,x)\\
&\partial_t I_h(t,x)=\eps E_h(t,x)\\
&\partial_t S_v(t,x)= \nabla \cdot(D(x) \nabla S_v(t,x)) -\beta_h(x)S_v(t,x)I_h(t,x)\\
&\partial_t I_v(t,x)= \nabla\cdot(D(x)\nabla I_v(t,x)) +\beta_h(x)S_v(t,x)I_h(t,x),
\end{align*}
with the following boundary and initial conditions:
\begin{align*}
&\partial_n(S_v(t,x))=\partial_n(I_v(t,x))=0 \qquad \text{ for all } t>0, x \in \partial \O\\
&(S_h(0,x),E_h(0,x),I_h(0,x), S_v(0,x),I_v(0,x))=(S_h^0(x),E_h^0(x),I_h^0(x), S_v^0(x),I_v^0(x)) \qquad \text{ for all }  x \in  \O.
\end{align*}

We can easily check that the above analysis does not significantly change with this model. Namely, all the proofs can be adapted to this model with minor changes. In particular, the proof of the exponential convergence to the equilibrium can be transposed readily since it is only based on fundamental properties of elliptic and parabolic equations that are satisfied by the new model. 

\section{Numerical study of the transitive phase}\label{sec:AS}

As a complement to the preceding study about equilibrium states, we implement in this section a global sensitivity analysis (GSA) to investigate how input parameters influence the variability of the transitive phase of the dynamics of infected hosts and infected vectors. 
For this analysis of the transitive phase, we performed numerous simulations of the multi-annual model $\mathcal M_2$ of Section \ref{sec:multi-annual.model} with spatially constant parameters $\beta_v$, $\beta_h$ and $D$. 
These simulations were specified by using the dynamics of {\it Xylella fastidosia} in southeastern mainland France as an inspiring example. {\it Xylella fastidosia} is a phytopathogenic bacterium infecting a large class of plant species and vectored by multiple insects, including {\it Philaenus spumarius} that is present in France. This bacterium was detected in 2015 in Corsica island, France, and in southeastern mainland France \cite{soubeyrand2018}. In August 2019, a new strain in France, called {\it pauca}, was collected and identified from an olive tree near the Italian border. Thus, inspired by the occurrence of this new strain and its potential spread, we specified the simulations of model $\mathcal M_2$ such as (i) the initial condition corresponds to an introduction in southeastern France near the Italian border, and (ii) the eventual spread of the pathogen occurs in the spatial domain formed by the French departments close to the Mediterranean sea where the conditions are relatively favorable for {\it Xylella fastidosia} expansion \cite{godefroid2019,martinetti2019}. 
Figure \ref{map-omega-intro} shows the study domain $\Omega$ and the introduction point used for all the simulations.



 \begin{figure}[!ht]
 \begin{center}
\includegraphics[width=8cm,trim={2.5cm 2cm 2.5cm 2cm},clip]{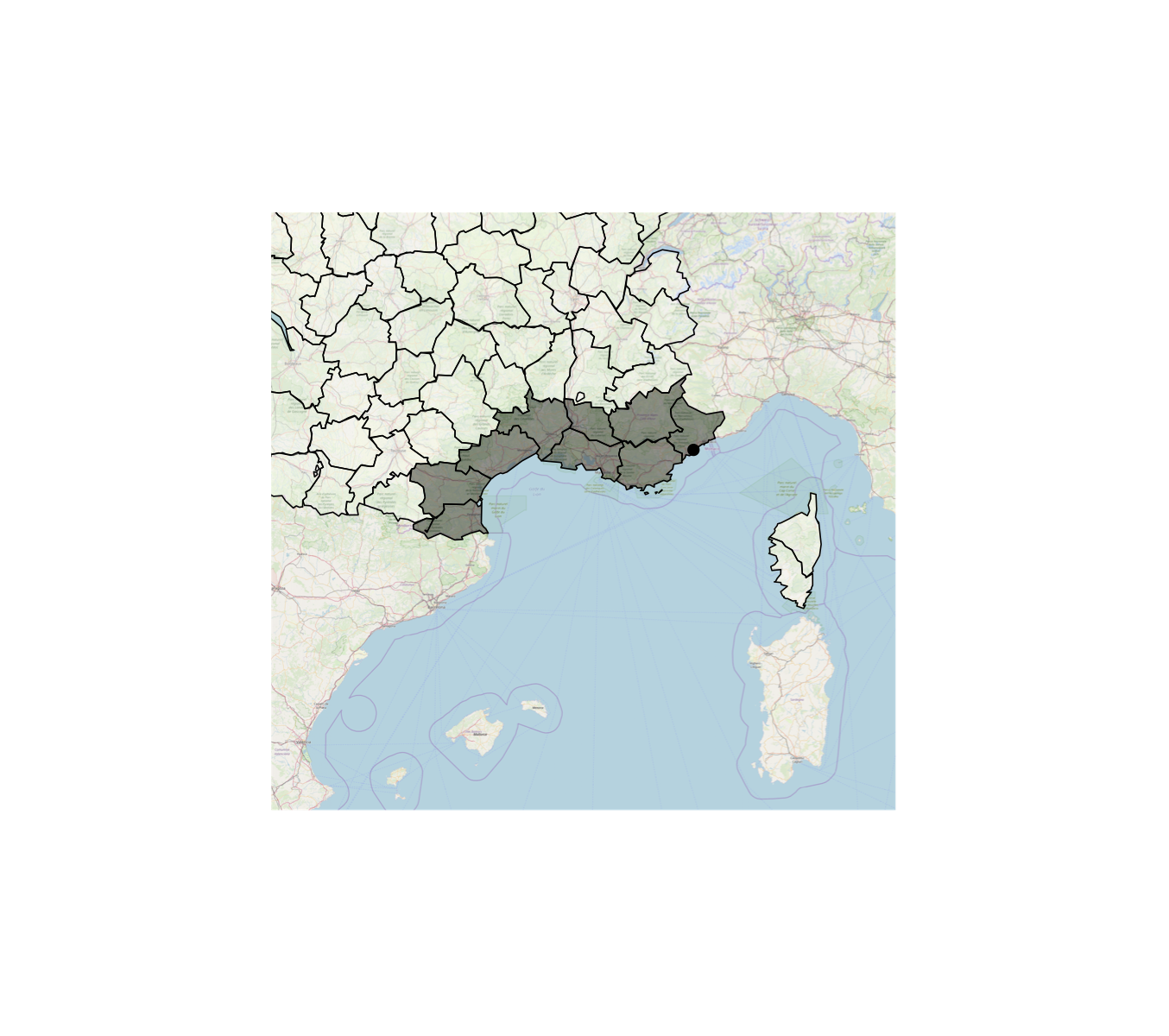}\\
\caption{Map of $\Omega$ (union of shaded French departments) and location of the introduction point (black dot) in southeastern France, near the Italian border, used for the simulation study. The borderline of $\Omega$ was regularized for the resolution of the system of equations.}
\label{map-omega-intro}
\end{center}
\end{figure}

\subsection{Two-stage resolution of the coupled partial differential equations}

Given the non linearity of the reaction terms in the modeling of the
pathogen spread, these terms were separated from the
diffusion terms in the resolution of the system of equations, using the operator splitting method. Thus, at a given time step, the equation system was resolved in two stages. The first stage concerning the reaction part of the model was implemented with the Newton-Raphson method. Since
vector diffusion is not accounted for in this first
stage, the partial differential equation system is simply an
ordinary differential equation system. In the second stage, the results from the first stage were used as initial conditions and the diffusion part of the model was handled with the finite element method. Computations have been performed with the \texttt{FreeFem++} software \cite{freefem}.

\subsection{Sensitivity analysis: methods}

As explained above, we used GSA to have a better understanding
of the evolution of the disease dynamics from its introduction and before it reaches its equilibrium state, in both the host and the vector compartments.
The objective of sensitivity analysis, in general, is to determine how variation in the model output depends upon the input information fed into the model. This way of reasoning results from the fact that input information, typically the values of model parameters, are generally uncertain.  
In GSA, one attempts to highlight a hierarchy between the uncertainty in the input factors or parameters with respect to the uncertainty of model outputs \cite{saltelli2000} (thereafter, the term {\it parameters} designate both factors and parameters). The deficit of knowledge on input parameters is described by uniform probability laws defined, without loss of generality, over $[0,1]$ for each parameter. Note that we assume the independence of parameter uncertainties. Let $\mathcal{D} = [0,1]^K$ denote the domain of uncertainty of the parameters, where $K$ is the number of parameters. 

We used a GSA method based on {\it variance decomposition} called ANOVA. Let $\phi:{\bf p} = (p_1,..., p_K) \in {\mathcal D}\to \phi({\bf p}) \in\R $. The quantity $\phi({\bf p})$ is a model output (e.g. the density of infected hosts at a given location and a given time) when the vector of parameters takes the value ${\bf p}$. The uncertainty of the model output (resulting from the uncertainty of the parameters) is defined  by its variance $V(\phi)$ satisfying $V(\phi) = \int_{\mathcal{D}} (\phi ({\bf p}) - \bar{\phi})^2 d{\bf p} $, where $\bar\phi=\int_{\mathcal{D}} \phi ({\bf p})\, d{\bf p}$ is the mean  of $\phi$. For $u \subseteq \mathcal{I}=\{1,...,K\}$, we denote $u^c$ the complement of $u$ in $\mathcal{I}$
and $\mathcal{P}(  \mathcal{I})$ the power set of $\mathcal I$ (i.e., the set of all the subsets of $\mathcal I$ including the empty set $\varnothing$). The {\it variance decomposition} of $\phi$ is defined by (see \cite{owen2013}): 
\begin{equation}
 \phi ({\bf p}) = \displaystyle  \sum_{u \in \mathcal{P}(  \mathcal{I})} \phi_u ({\bf p}), \label{decomp}
\end{equation}
where $\phi_u$ is defined by:
\begin{equation}
 \phi_u({\bf p})=\displaystyle \int \phi({\bf p}) d{\bf p}_{u^c} - \displaystyle \sum_{v \in \mathcal{P}(u),v\neq u}\phi_v({\bf p})
 \end{equation}
 with $d{\bf p}_u=\displaystyle \prod_{i \in u} dp_i$ according to the independence hypothesis on parameter uncertainties.

It follows that $\bar{\phi}= \phi_{\varnothing}({\bf p})$. When $u=\{i\}$, one obtains the main effect of parameter $p_i$, which is 
$$ \phi_i({\bf p})=\int \phi({\bf p}) d{\bf p}_{i^c} - \bar{\phi}$$ 
that depends on $p_i$.  When $u=\{i,j\}$, one obtains the interaction of order two of $p_i$ and $p_j$, which is  
$$\phi_{i,j}({\bf p})=\int \phi({\bf p}) d{{\bf p}}_{\{i,j\}^c}- \phi_i({\bf p}) -\phi_j({\bf p}) -\bar{\phi}$$
that depends on $p_i$ and $p_j$. Moreover, we have the properties  
$\int_{\mathcal{D}} \phi_{u}({\bf{p}}) \, d{\bf p} =0 $ for $u\neq\varnothing$ and $\int_{\mathcal{D}} \phi_{u}({\bf p}) \phi_{v}({\bf{p}})\,d{\bf p}=0$ for $u \neq v$. It follows that 
$$ V(\phi)=\displaystyle \sum_{u\in \mathcal{P}(  \mathcal{I})} V(\phi_u). $$ 
The principal sensitivity index ($\text{PI}_i$) of parameter $p_i$ is defined by: 
 $$\text{PI}_i = \frac{V(\phi_i)}{V(\phi)} \in [0,1]. $$
The total sensitivity index  ($\text{TI}_i$) of parameter $p_i$ is defined by
 $$\text{TI}_i=  \displaystyle \Big( \displaystyle \sum_{v \; \cap \; \{i\} \neq \varnothing }  V(\phi_{v \; \cap \;\{i\}}) \Big) /{V(\phi)} \in [0,1].$$
These indices satisfy the property $\text{TI}_i \ge \text{PI}_i$. A large value of $\text{TI}_i$ with respect to $\text{PI}_i$ indicates the presence of interaction (of any order) between $x_i$ and other parameters.
 
The numerical challenge of GSA is to compute these sensitivity indices (SIs) with Monte-Carlo simulations. This challenge was tackled by using the latin hypercube square method for sampling parameters and following the approach proposed by Monod {\it et al.} \cite{Monod2006} for computing the sensitivity indices. This approach requires $M \times (2 K + 1)$ evaluations of the model with an initial sampling scheme of $M$ different points in the parameter domain $\mathcal D$ (of dimension $K=3$), and we used $M=300$ in the application. Thus, the model was run 2100 times over a period of 100 years.
In practice, the sampling method has been implemented by using the package \verb"lhs" \cite{lhspackage} and SIs have been computed by using the function \verb"sobolEff()" from the package \verb"sensitivity" \cite{sensitivitypackage} in the \verb"R" environment programming software \cite{Rsoftware}. 
Variation ranges of $\beta_h$, $\beta_v$ and $D$ are given in Table \ref{tabIntervals}. In addition, we set $\eps=0.02$, $S_v^0\equiv 300$, $I_v^0\equiv 0$, $S_h^0\equiv 300$, $E_h^0 \equiv 0$ and $I_h^0 \equiv 0$.

\begin{table}
\begin{center}
\begin{tabular}{cc}
\hline
Parameter & Variation range \\
\hline
$\beta_v$ & [5,25] \\
$\beta_h$ & [5,25]\\
$D$ & [5,15000] \\
\hline
\end{tabular}
\caption{Variation ranges of input parameters for the global sensitivity analysis of model $\mathcal{M}_2$. }
\label{tabIntervals}
\end{center}
\end{table}

\subsection{Sensitivity analysis: results about infected hosts and vectors}

Figures~\ref{meanvarHosts} and \ref{meanvarVectors} in Supplementary Material map the mean and standard deviation (std.) of the numbers of infected hosts and infected vectors, respectively, computed from the 2100 simulations. Mean and std. have been computed at 600 points in the domain $\O$, and a smoothing procedure has been applied to build maps; see figure captions. These figures illustrate the convergence to an equilibrium, in which all hosts and vectors are infected. Equilibrium is almost achieved at $t=50$ years for infected vectors, whereas infected hosts approach equilibrium beyond $t=100$ years (the value of the plateau, namely 300, is not yet reached at $t=100$; see Figure \ref{MeanVarHoststempslongs}). Standard deviations for infected hosts and infected vectors follow similar evolution. In the area surrounding the site of introduction, std. is relatively large at the beginning of the epidemic and then decreases with time. In further areas from the site of introduction, std. is low at the beginning, then increases progressively and finally decreases. We can further note that the peak of std. is higher for infected vectors than for infected hosts, and std. is more uniform for infected hosts than for infected vectors, especially over the first 50 years. 
Thus, the propagation of the disease in the host and vector compartments clearly differ and the results of the GSA below will highlight the main drivers of the propagation variability.


The GSA was performed to assess the influence across space and time of three input parameters $\beta_h$, $\beta_v$ and $D$, related to disease transmission and diffusion, on two outcome variables: the infected hosts $I_h(t,x)$ and the infected vectors $I_v(t,x)$. The initial conditions and the parameter $\eps$ were not included in this analysis. We can however point out that, based on a preliminary study not shown here, $\eps$ (related to the latency period in the host) has a strong effect obscuring the effects of the other parameters as soon as its range of variation is relatively large. The strong effects of $\eps$ corroborates results provided by \cite{rimbaud2018} with a stochastic epidemic model. Since $\beta_h$ and $\beta_v$ have a symmetric role in the model for $I_h(t,x)$ and $I_v(t,x)$ and to avoid to distort the sensitivity analysis with respect to these parameters, their variation ranges have been set up equal (see Table \ref{tabIntervals}). In addition, for facilitating the relative interpretation of the effects of parameters on $I_h(t,x)$ and $I_v(t,x)$, we considered a situation where the densities of hosts and vectors are initially constant and both equal to 300 units across space. 

Figures~\ref{IndiceP-Hosts} and~\ref{IndiceP-Vectors} display PIs of parameters across space and time (up to $t=50$ years) for infected hosts and infected vectors, respectively (Supplementary Figure \ref{IndiceTempsLongs} provide PIs for infected hosts between year 70 and year 100). To complete the interpretation of these figures, we recorded the spread of epidemics along a 40-points transect  (shown by Figure \ref{transect}) going through the study region from the introduction point. Then, simulations have been grouped by class of interval of input parameters. For each parameter, four equal-length classes have been created by dividing the interval of simulation defined in Table \ref{tabIntervals}. Figures~\ref{transectHote} and~\ref{transectVector} give the means of $I_h$ and $I_v$, respectively, for each parameter class, at different dates and along the transect.

Parameters $\beta_v$ and $\beta_h$ broadly play similar roles in the spatio-temporal variability of the number of infected hosts, even if $\beta_h$ is more influential far from the introduction site at the early stage of the epidemics (Figure~\ref{IndiceP-Hosts}). In contrast, $D$ mostly plays a minor role, except after 30-50 years in areas far from the introduction point. This is corroborated by Figure~\ref{transectHote} illustrating the impact of parameter variations on infected-hosts variations along the above-mentioned transect. Thus, using as levers the reduction of $\beta_v$ (e.g., by {\it inciting} vectors to feed on non-host plants via the planting of such plants or the settling of repellents/attractors) or the reduction of $\beta_h$ (e.g., by protecting host plants with insect-proof nets) should broadly have the same significant impact on the number of infected hosts across space. However, attempting to hamper the diffusion of insects is not expected to greatly decrease the number of infected hosts, at least at the early stage of the epidemics.

Non-intuitively, the spatio-temporal variability of the number of infected vectors has not the same drivers than the spatio-temporal variability of the number of infected hosts: the PIs for infected hosts and infected vectors clearly differ after year 10, with a prominent effect of $\beta_v$ in the variability of the number of infected vectors, as shown by Figures~\ref{IndiceP-Vectors} and~\ref{transectVector}. Once the hosts of a given area are infected at a sufficiently large proportion, the variability in $I_v$ mostly depends on the contact rate of a vector with hosts (i.e., $\beta_v$). This situation holds near the introduction points at years 20-30 and over the whole study region at year 50. Thus, on a long term, the decrease of the number of infected vectors requires above all a reduction of $\beta_v$  (e.g., by {\it inciting} vectors to feed on non-host plants).


 \begin{figure}[!ht]
 \begin{center}
\includegraphics[width=14cm]{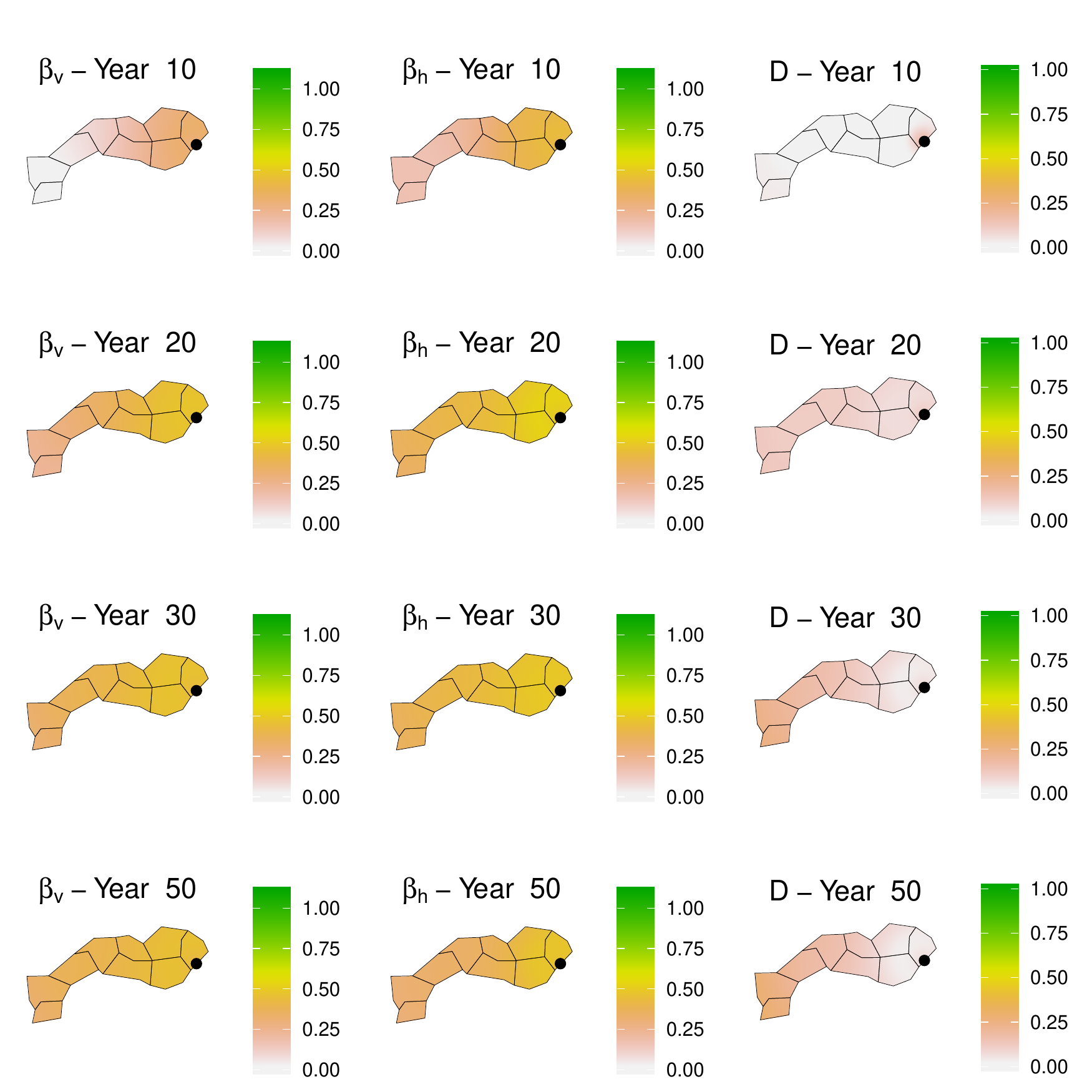}\\
\caption{Maps of the principal sensitivity indices (PIs) for the number of  infected hosts $I_h(t,x)$. From left to right, columns correspond to parameters $\beta_v$, $\beta_h$ and $D$.  From top to bottom, rows correspond to years 10, 20, 30 and 50. Maps were obtained from the 600 PI values scattered in $\Omega$ via the linear interpolation implemented in the \texttt{interp} function of the \texttt{R} package \texttt{akima} \cite{akima}. }
\label{IndiceP-Hosts}
\end{center}
\end{figure}

\begin{figure}[!ht]
\begin{center}
\includegraphics[width=14cm]{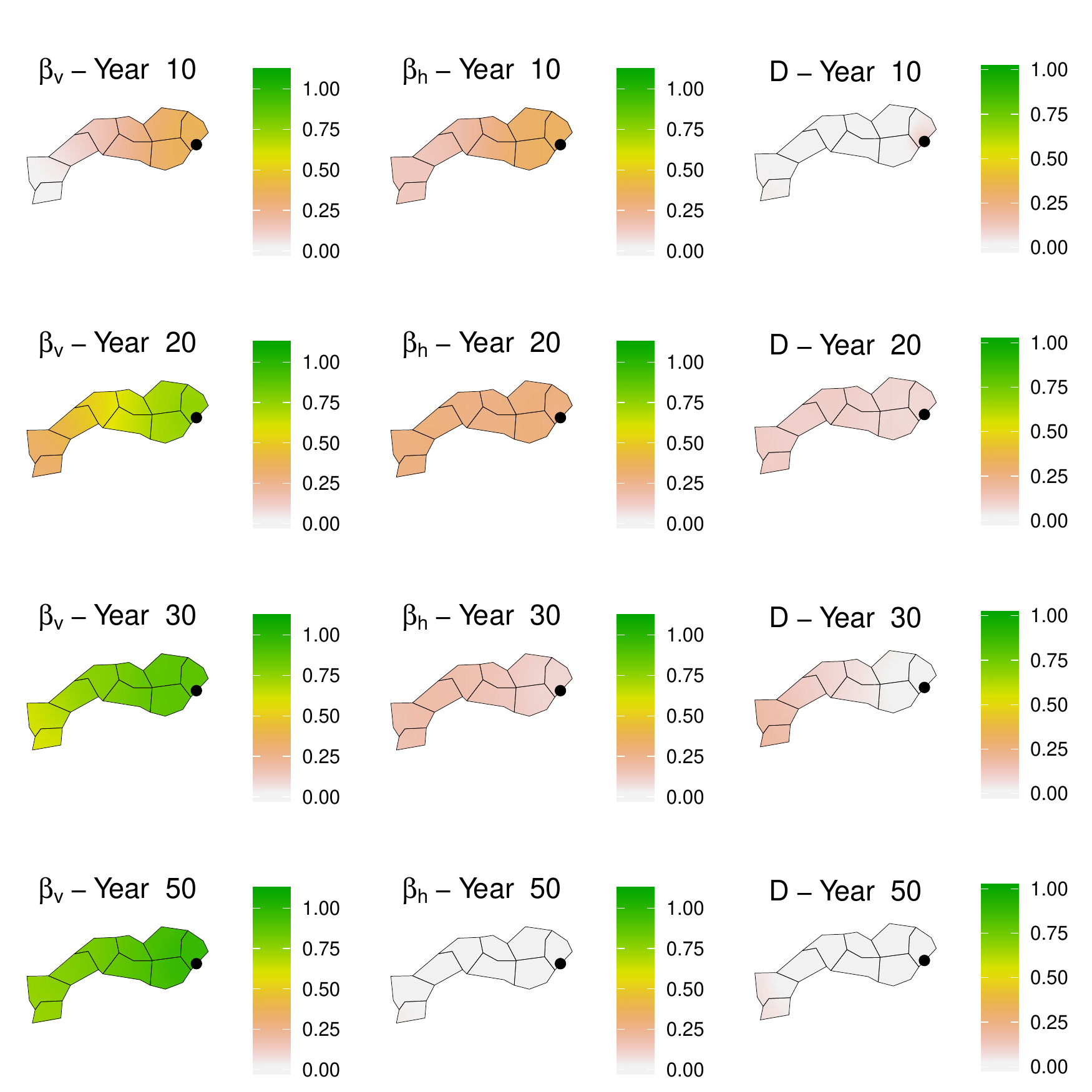}\\
\caption{ Maps of the principal sensitivity indices (PIs) for the number of infected vectors $I_v(t,x)$. From left to right, columns correspond to  parameters $\beta_v$, $\beta_h$ and $D$.  From top to bottom, rows correspond to years 10, 20, 30 and 50. Maps were obtained from the 600 PI values scattered in $\Omega$ via the linear interpolation implemented in the \texttt{interp} function of the \texttt{R} package \texttt{akima} \cite{akima}.}
\label{IndiceP-Vectors}
\end{center}
\end{figure}


\begin{figure}[!ht]
\begin{center}
\includegraphics[width=8cm]{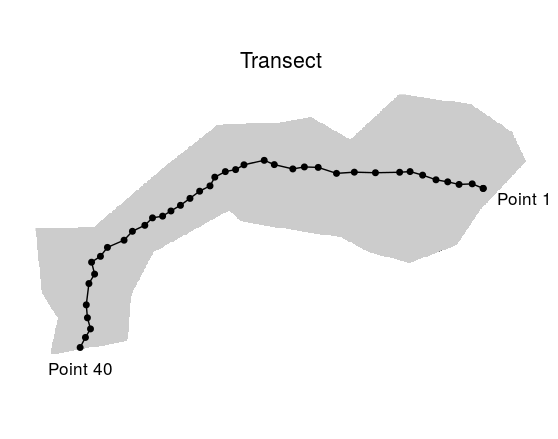}
\caption{Plot of the 40 points transect going through the study region. Points are numbered from 1 (near the site of disease introduction) to 40 (far away from the site of introduction).  }
\label{transect}
\end{center}
\end{figure}

\begin{figure}[!ht]
\begin{center}
\includegraphics[width=14cm]{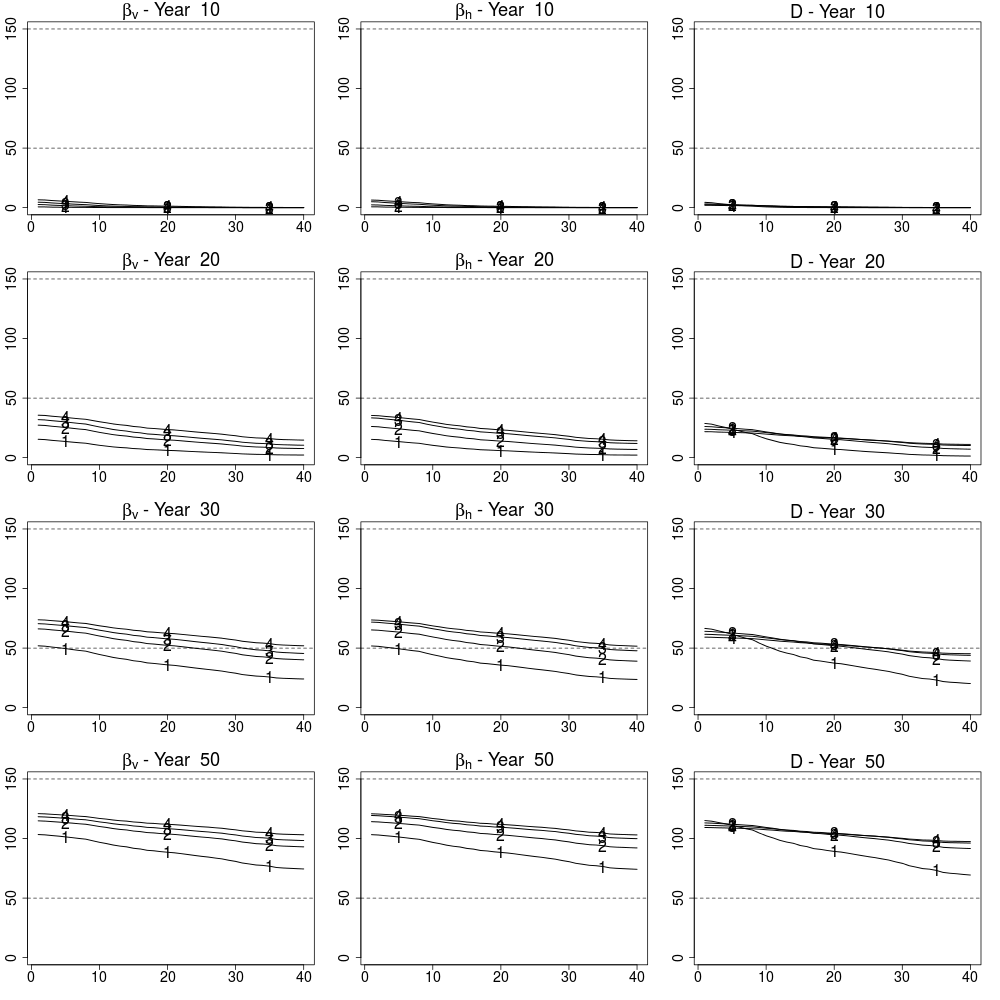}
\caption{Means of $I_h(t,x)$ for 40 points $x$ along the transect shown in Figure~\ref{transect} and for years 10, 20, 30 and 50. Classes are numerated from 1 to 4. Class 1 corresponds to low values of parameters and class 4 to high values. X-axis gives the point number  along the transect and the Y-axis gives the mean of $I_h(t,x)$.   }
\label{transectHote}
\end{center}
\end{figure}

\begin{figure}[!ht]
\begin{center}
\includegraphics[width=14cm]{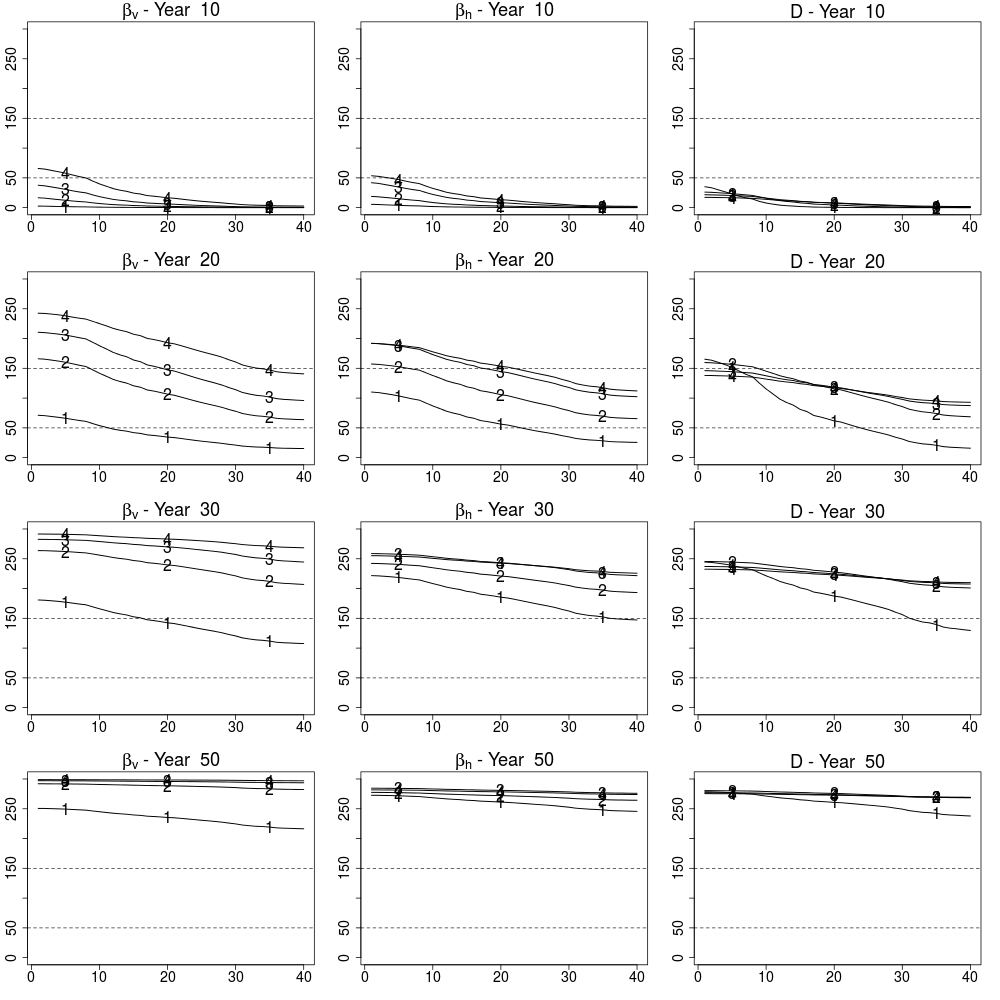}\\
\caption{Means of $I_v(t,x)$ for 40 points $x$ along the transect shown in Figure~\ref{transect} and for years 10, 20, 30 and 50. Classes are numerated from 1 to 4. Class 1 corresponds to low values of parameters and class 4 to high values. X-axis gives the point number along the transect and the Y-axis gives the mean of $I_v(t,x)$.}
\label{transectVector}
\end{center}
\end{figure}

\newpage
\section{Discussion}\label{sec:discu}

We investigated the spatiotemporal dynamics of a vector-borne disease by means of equilibrium and sensitivity analyses of an explicit host-vector, spatiotemporal, compartmental model. The dynamics of vector-borne diseases are relatively complex because of interactions between host and vector compartments as well as interactions between parameters governing fluxes between compartments. This article disentangles a part of this complexity.


In a first approach, we identified theoretical non-negative equilibrium in two contexts: (i) when the vector population is considered as permanent, and (ii) when the vector population has a cyclic annual dynamics consisting of an emergence stage at the beginning of the year, a mortality stage at the end of the year and no adult-to-offspring transmission of the pathogen from one year to the following one.  In both contexts, the non-negative equilibrium implies the infection of the whole host population. We also provided a quantitative bound of convergence to the equilibrium, giving a first estimate on the speed of total contamination.

In a second approach, we considered the model built in context (ii) and explored the impact of parameters related to transmission and diffusion on the transitive spatiotemporal patterns of infected hosts and vectors. We pointed out similar influences of the contact rates `of a host with vectors' and `of a vector with hosts' ($\beta_h$ and $\beta_v$) on the spatiotemporal variability of the number of infected hosts $I_h$. This similarity indicates that one has two levers of action with similar {\it expected efficiency} for slowing down the propagation of the disease in the host population: the lever on $\beta_v$ that can be activated, e.g., by {\it inciting} vectors to feed on non-host plants via the planting of such plants or the settling of repellents/attractors; the lever on $\beta_h$ that can be activated, e.g., by protecting host plants with insect-proof nets. In a case such as the dynamics of {\it Xylella fastidiosa} in southeastern France, it would be interesting in a further study to explicitly incorporate into the model different management measures (on both hosts and vectors) and their individual costs, in the aim of guiding decision makers based on an economic-epidemiological analysis as in \cite{fabre2019,picard2019,rimbaud2019}. 

The numerical and sensitivity analyses also highlighted the trend of vectors to be beyond the front line of the epidemics in the host population. This trend is obvious since hosts are fixed whereas vectors are mobile in our model. However, our model could be used to quantitatively design monitoring strategies of the epidemics targeting the vectors beyond the front line, in the aim of detecting cryptic disease foci on hosts and anticipating the future spread.

The type of model analyses that we carried out is a first step to understand the main factors driving epidemics generated by outbreak models we are interested in. The sensitivity analysis  of the model combined with its analytic study provide some valuable insights on which components significantly affect the epidemics and how they affect them. Such insights may be crucial to point out model components on which epidemiologists should improve knowledge and whose mathematical formalization should be refined to gain in model realism. Such model modifications are inevitably inherent to the pathogenic system considered and may for example lead to the introduction of more accurate  descriptions of some of the demographic processes and environmental effects. Typically, for {\it Xylella fastidiosa}, one could carry out further work using the model proposed here and enriched with recent feedback on the biology of this pathogen (e.g. about its sensitivity to precipitation and temperature  \cite{martinetti2019}).


\

\noindent{\bf Acknowledgements.}
This research was funded by the INRA-DGAL Project 21000679 and the HORIZON 2020 XF-ACTORS Project SFS-09-2016. We thank Claude Bruchou and Olivier Bonnefon, INRAE, BioSP, for discussions concerning the methodology of sensitivity analysis and the numerical resolution of partial differential equations.

\clearpage

\bibliographystyle{plain}
 \bibliography{xyllela.bib}

 \clearpage
\appendix 
 \section{Supplementary Material}
\counterwithin{figure}{section}
\setcounter{figure}{0}

\begin{figure}[!ht]
\begin{center}
\includegraphics[width=12cm]{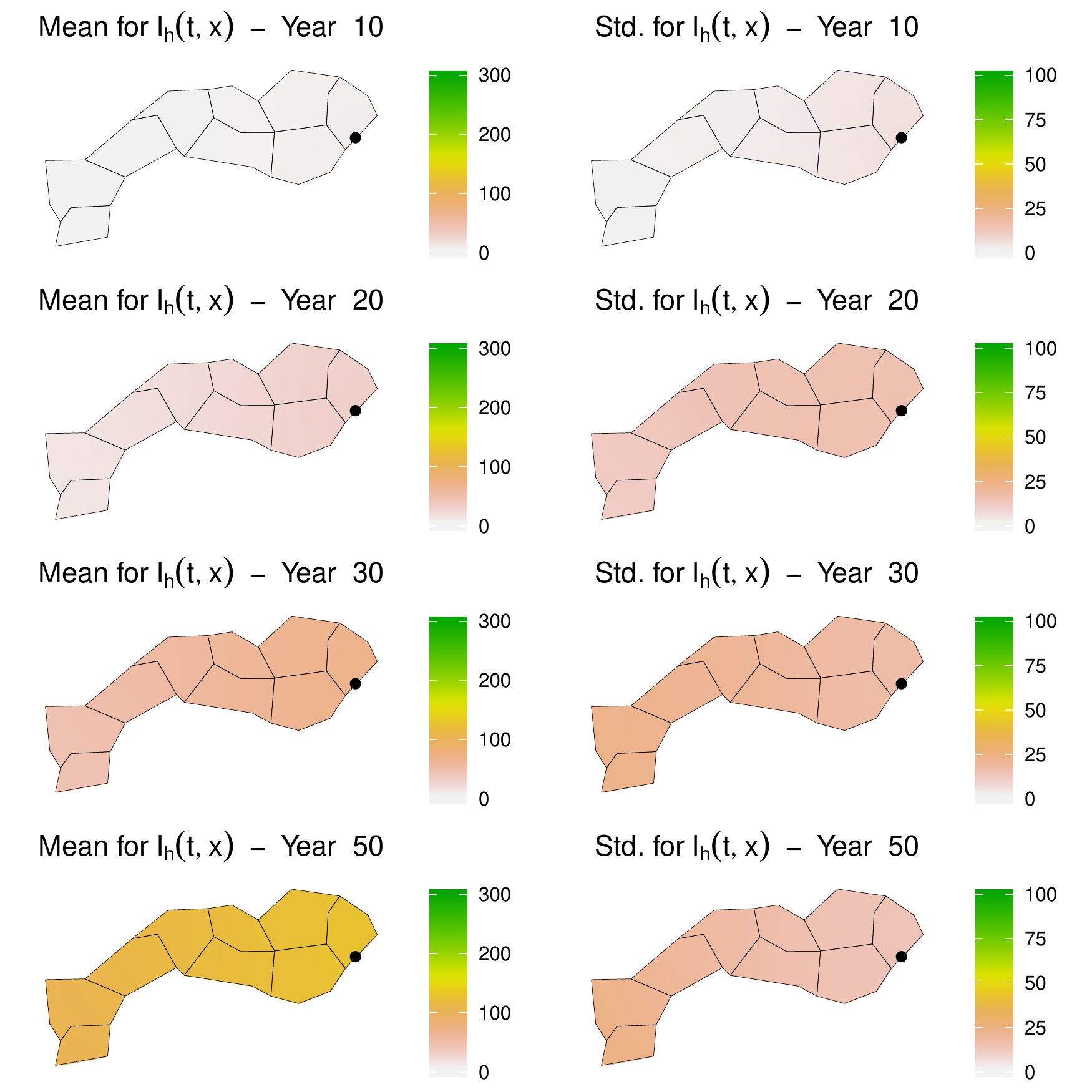}\\
\caption{Mean and standard deviation (Std.) of the number of infected vectors $I_h(t,x)$. The left column gives means at years 10, 20, 30 and 50. The right column gives standard deviations at the same time points. The black point corresponds to the site of introduction of the disease. Maps were obtained from the 600 values scattered in $\Omega$ via the linear interpolation implemented in the \texttt{interp} function of the \texttt{R} package \texttt{akima} \cite{akima}. }
\label{meanvarHosts}
\end{center}
\end{figure}

\begin{figure}[!ht]
\begin{center}
\includegraphics[width=12cm]{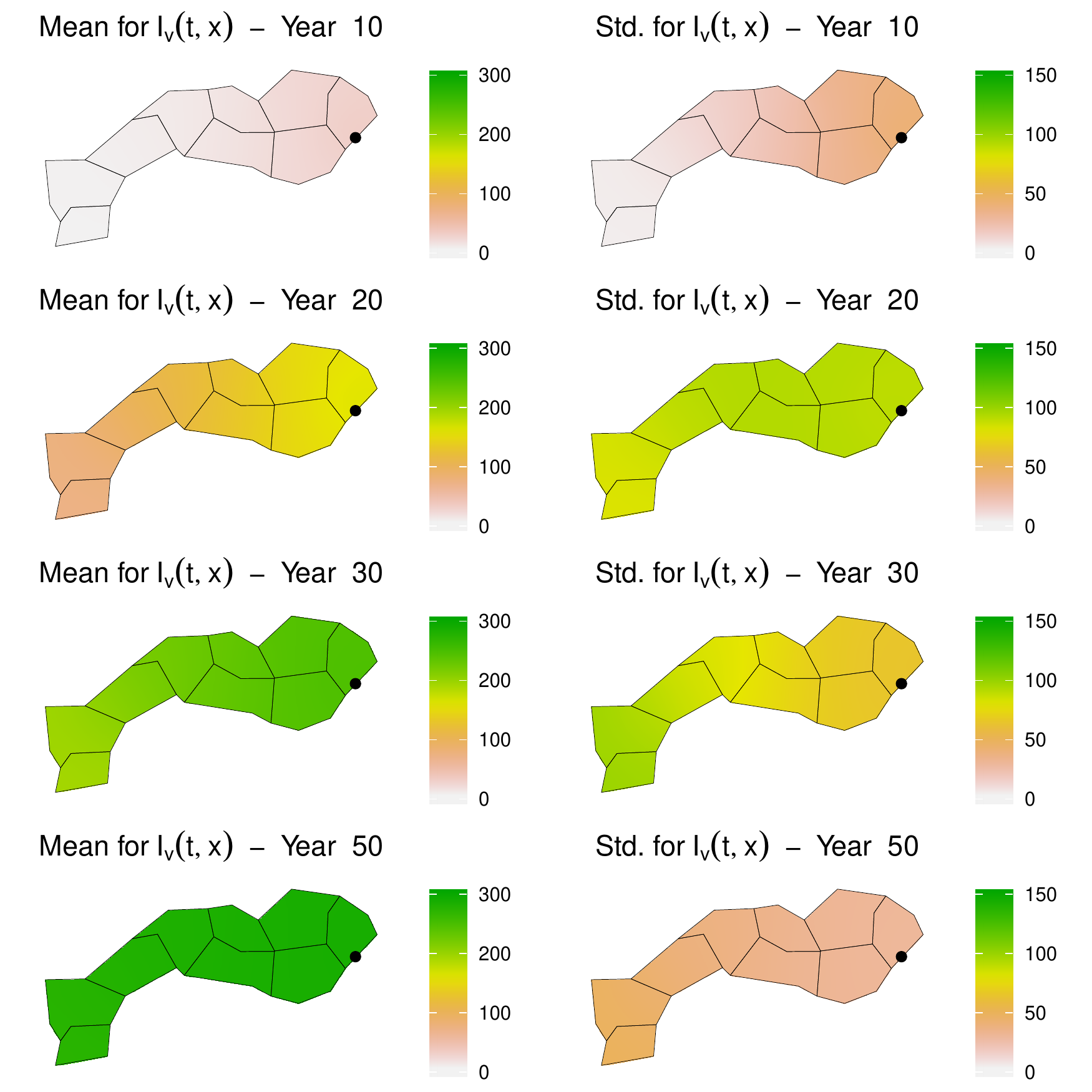}\\
\caption{Mean and standard deviation (Std.) of the number of infected vectors $I_v(t,x)$. The left column gives means at years 10, 20, 30 and 50. The right column gives standard deviations at the same time points. The black point corresponds to the site of introduction of the disease. Maps were obtained from the 600 values scattered in $\Omega$ via the linear interpolation implemented in the \texttt{interp} function of the \texttt{R} package \texttt{akima} \cite{akima}.} 
\label{meanvarVectors}
\end{center}
\end{figure}

\begin{figure}[!ht]
\begin{center}
\includegraphics[width=12cm]{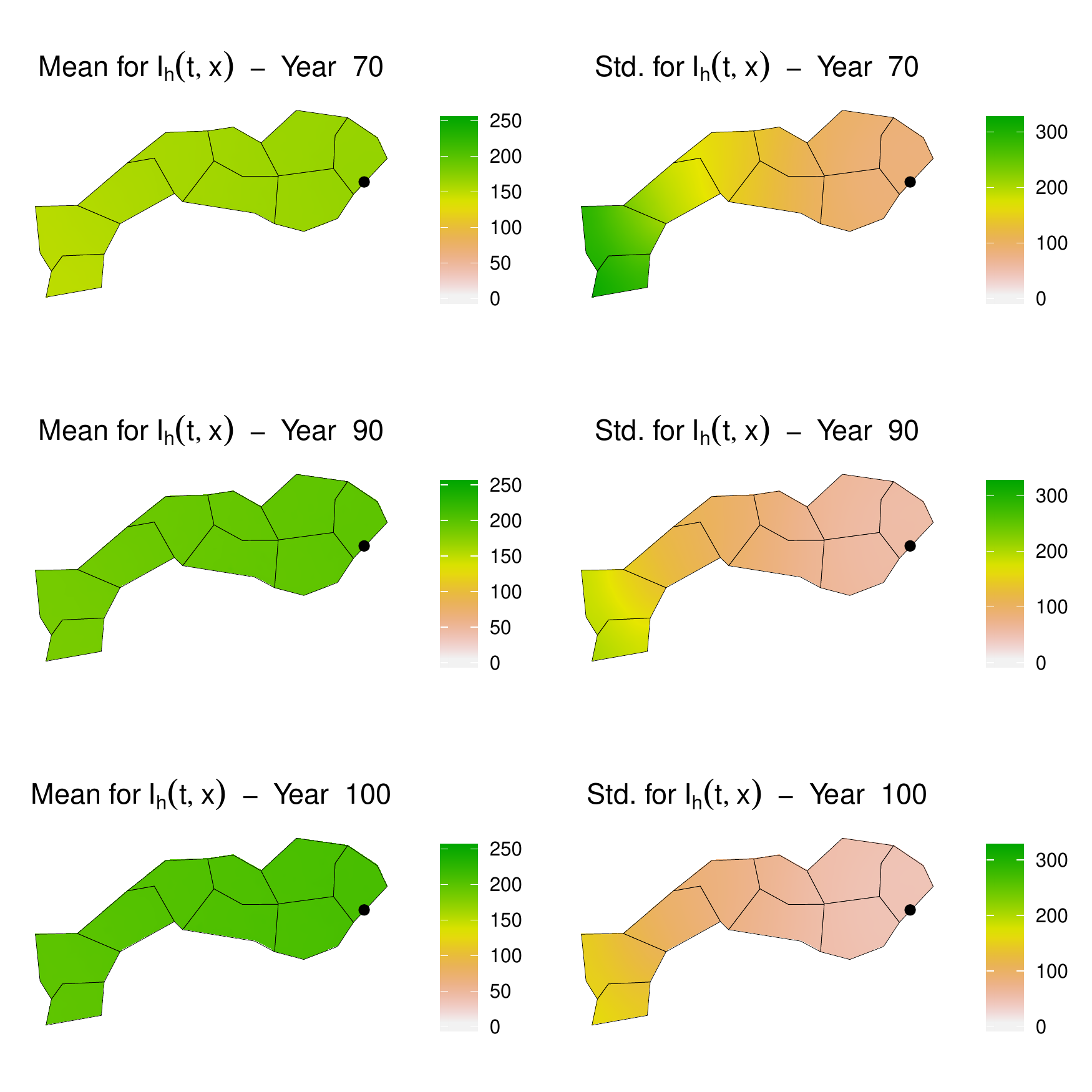}\\
\caption{Mean and standard deviation (Std.) of the number of infected vectors $I_h(t,x)$. The left column gives means at years 70, 90 and 100. The right column gives standard deviations at the same time points. The black point corresponds to the site of introduction of the disease. Maps were obtained from the 600 values scattered in $\Omega$ via the linear interpolation implemented in the \texttt{interp} function of the \texttt{R} package \texttt{akima} \cite{akima}.} 
\label{MeanVarHoststempslongs}
\end{center}
\end{figure}

\begin{figure}[!ht]
\begin{center}
\includegraphics[width=12cm]{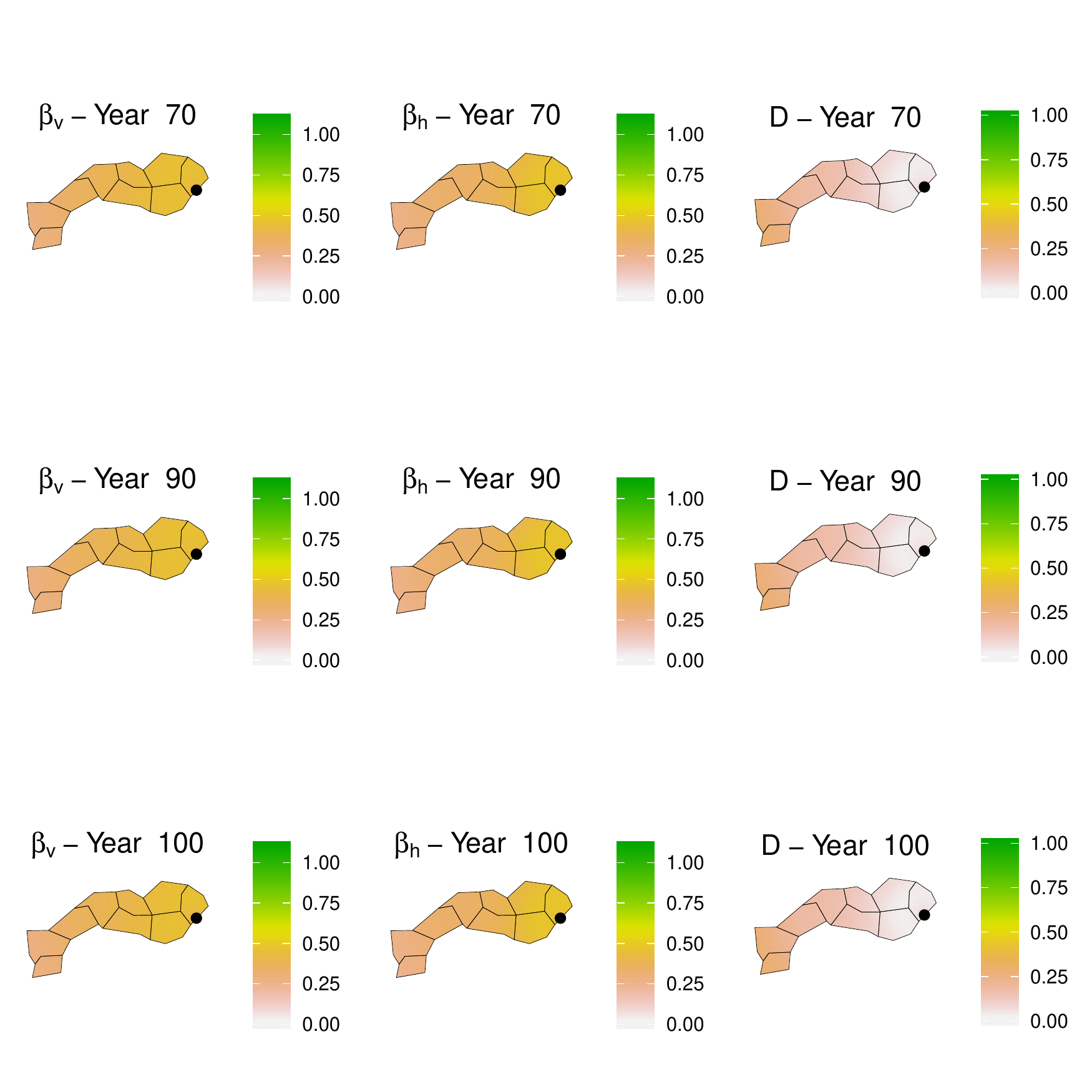}\\
\caption{Maps of the principal sensitivity indices (PIs) for the number of  infected hosts $I_h(t,x)$. From left to right, columns correspond to parameters $\beta_v$, $\beta_h$ and $D$.  From top to bottom, rows correspond to years 70, 90 and 100. Maps were obtained from the 600 values scattered in $\Omega$ via the linear interpolation implemented in the \texttt{interp} function of the \texttt{R} package \texttt{akima} \cite{akima}.} 
\label{IndiceTempsLongs}
\end{center}
\end{figure}

\clearpage

\section{Proofs}
In this appendix we give the proof of Proposition \ref{prop1}.
\begin{proof}
The search of positive equilibria of System $\mathcal{M}_1$ implies to look for positive solutions to the following set of equations:
\begin{align}
&\bar \beta_v(x)S_h^*(x)i_v^*(x)=0\qquad &\text{ for all }  x \in  \O. \label{equi1}\\
&E_h^*(x)=0\qquad &\text{ for all }  x \in  \O.\label{equi2}\\
&D(x) \Delta s_v^*(x) -\beta_h(x)s_v^*(x)I_h^*(x)=0\qquad &\text{ for all }  x \in  \O.\label{equi3}\\
&D(x) \Delta i_v^*(x) +\beta_h(x)s_v^*(x)I_h^*(x)=0\qquad &\text{ for all }  x \in  \O.\label{equi4}\\
&\partial_ns_v^*(x)=\partial_ni^*_v(x)=0 \qquad &\text{ for all }  x \in \partial \O.\label{equi5}
\end{align}
By integrating over the domain $\O$ the  PDE \eqref{equi4} with respect to the measure $d\mu(x)=\frac{dx}{D(x)}$ we then see that 
$$\int_{\O}\beta_h(x)S_v^*(x)I_h^*(x)\,d\mu(x) =0,$$
which  enforces that for almost every $x \in \O$ 
\begin{equation}\label{equi6}
\beta_h(x)S_v^*(x)I_h^*(x)=0
\end{equation}
since $\beta_h,S_v^*$ and $I_h^*$ are non negative quantities.

As a consequence, $i^*_v$ and $s^*_v$ satisfy the following PDE
\begin{align}
&\Delta u(x)=0\qquad &\text{ for almost every }  x \in  \O,\\
&\partial_n u(x)=0 \qquad &\text{ for all }  x \in \partial \O,
\end{align}
which in turn implies that $i_v\equiv C_0\ge 0$ and $s_v\equiv C_1\ge 0$  which  thanks to \eqref{invar2} must satisfy 
\begin{equation}\label{compat}
(C_0+C_1)|\O|_{\mu}=C^*.
\end{equation} 
Going back to Equation  \eqref{equi1}, we see that 
$$C_0\bar \beta_v(x)S_h^*(x)=0 \qquad \text{ for all }  x \in  \O.$$
Since $\bar \beta_v>0$ the later  induces a simple dichotomy with respect to $C_0$:
\begin{itemize}
\item Either $C_0=0$ and then from \eqref{compat}  $D(x)S_v=s_v^*\equiv C_1=\frac{C^*}{|\O|_{\mu}}>0$, 
 which in turn implies that 
$I^*_h=0$ almost everywhere thanks to \eqref{equi6}.
Therefore thanks to \eqref{invar1}, we get $S_h^*(x)=N(x)$ for almost every $x$.
Thus we get $(N(x),0,0,\frac{C^*}{D(x)|\O|_{\mu}},0)$ for the first equilibrium of the system \eqref{eq1}--\eqref{eq5} with boundary condition \eqref{eq-bc}.
\item Or $C_0>0$, then in this situation from \eqref{equi1} we get $S_h^*(x)=0$ almost everywhere and by \eqref{invar1} and  \eqref{equi6} we get  $I^*_h=N(x)$ and $s_v^*\equiv 0$.
Thus  we get $(0,0,N(x),0,\frac{C^*}{D(x)|\O|_{\mu}})$ for the second equilibrium of the system \eqref{eq1}--\eqref{eq5} with boundary condition \eqref{eq-bc}.
\end{itemize}

Let us now check the stability of these positive equilibrium. Note that since $S_h$ and $I_h$ are respectively  decreasing and increasing, we can readily claim that the stationary state $(N(x),0,0,\frac{C^*}{D(x)|\O|_{\mu}},0)$ is unstable. Let us now check the local stability of the endemic state $ES=(0,0,N(x),0,\frac{C^*}{D(x)|\O|_{\mu}}) $. To do so, we linearize the system around $ES$, which gives 
$$
\text{Jacobian}(ES):=\begin{pmatrix}
-\bar \beta_v\frac{C^*}{|\O|_{\mu}} & 0     & 0 & 0 & 0\\
\bar \beta_v\frac{C^*}{|\O|_{\mu}}  & -\eps & 0 & 0 & 0\\
0           & \eps  & 0 &0  & 0\\
0           & 0     & 0 & D(x)\Delta   -\beta_h N(x) & 0   \\
0           & 0     & 0 & +\beta_h N(x) & D(x) \Delta )   
\end{pmatrix}
$$
 and search for the sign of the largest eigenvalue. Observe that up to a change of basis we can rewrite the Jacobian matrix $\text{Jacobian}(ES)$ as follows: 
$$
\text{Jacobian}(ES)=\begin{pmatrix}
-\bar \beta_v\frac{C^*}{|\O|_{\mu}} & 0     & 0 & 0 & 0\\
0           & -\eps & 0 & 0 & 0\\
0           & 0     & 0 & 0 & 0\\
0           & 0     & 0 & D(x)\Delta  -\beta_h N(x) & 0   \\
0           & 0     & 0 & +\beta_h N(x) & D(x) \Delta   
\end{pmatrix}.
$$
Therefore the stability of the endemic state $(ES)$ is then only defined  by the right below block, that is 
$$
\text{Block}(ES)=\begin{pmatrix}
D(x) \Delta -\beta_h N(x) & 0   \\
 \beta_h N(x) &D(x) \Delta 
\end{pmatrix}
$$

Such a block is known to induce a negative spectral bound \cite{Gilbarg2001,Zeidler1986}. Thus, $(ES)$ is a locally stable equilibrium. From the monotone property of $I_h$ and $S_h$ we can also infer that the state $(ES)$ is indeed  globally stable. 
\end{proof}

Below, we establish the proof of Proposition \ref{prop3}.
\begin{proof}
A positive equilibrium $(S^*_h(t,x),E^*(t,x),I_h^*(t,x),S_v^*(t,x),I_v^*(t,x))$ of the impulsive system will then be a positive time periodic  solution of \eqref{eq-imp1}--\eqref{eq-imp5} of period $T$.   As a consequence, from the equation \eqref{eq-imp1}--\eqref{eq-imp3} we deduce  that $I_h^*(t,x)$ and $S_h^*(t,x)$ are respectively  a time increasing and a time decreasing periodic function. Thus $I_h^*(t,x)$ and $S_h^*(t,x)$ must be  independent of time and therefore   
\begin{align}
&\beta_v(x)S_h^*(x)I_v^*(t,x)=0\qquad &\text{ for all }  x \in  \O. \label{equi-imp1}\\
&E_h^*(x)=0\qquad &\text{ for all }  x \in  \O.\label{equi-imp2}\\
&\partial_t S_v^*(t,x) =\Delta (D(x)  S_v^*(t,x)) -\beta_h(x)S_v^*(t,x)I_h^*(x)\qquad &\text{ for all }  x \in  \O.\label{equi-imp3}\\
&\partial_t I_v^*(t,x) =\Delta(D(x) I_v^*(t,x)) +\beta_h(x)S_v^*(t,x)I_h^*(x)\qquad &\text{ for all }  x \in  \O.\label{equi-imp4}\\
&\partial_n(D(x)S_v^*(t,x))=\partial_n(D(x)I^*_v(t,x))=0 \qquad &\text{ for all }  x \in \partial \O.\label{equi-imp5}\\
&(S_v^*(0,x),I^*_v(0,x)=(S_0(x),0) \qquad &\text{ for all }  x \in  \O.\label{equi-imp6}
\end{align}

Note that since $S_0(x)$ is a  non negative function, a straightforward application of the parabolic maximum principle implies that for $t\in (0,T)$, the functions $D(x)S_v^*(t,x)$ and $D(x)I_v^*(t,x)$ are positive for all $x \in \O$ and $t>0$ and  so are $S_v^*(t,x)$ and $I_v^*(t,x)$ since $D(x)>0$.

 Thanks to \eqref{equi-imp1} and \eqref{invar3} this implies that  $S_h^*(x)\equiv 0$ and $I^*_h(x)=N(x)$ which in turns leads to  $S_v^*,I^*_v$ satisfying 
 \begin{align}
&\partial_t S_v^*(t,x) =\Delta (D(x) S_v^*(t,x)) -\beta_h(x)S_v^*(t,x)N(x)\qquad &\text{ for all } 0<t<T,\; x \in  \O\\\
&\partial_t I_v^*(t,x) =\Delta (D(x) I_v^*(t,x)) +\beta_h(x)S_v^*(t,x)N(x)\qquad &\text{ for all }  0<t<T,\;x \in  \O\\
&\partial_nS_v^*(t,x)=\partial_nI^*_v(t,x)=0 \qquad &\text{ for all } 0<t<T,\; x \in \partial \O\\
&(S_v^*(0,x),I^*_v(0,x)=(S_0(x),0) \qquad &\text{ for all }  x \in  \O.
\end{align}
\end{proof}
\end{document}